\documentclass[12pt,twoside,reqno]{amsart}
\usepackage{amsmath,amssymb,mathrsfs,graphicx,bbm,amsfonts,amsthm, bm}
\usepackage [latin1]{inputenc}
\usepackage{color}
\usepackage{todonotes}
\usepackage{comment}

\usepackage{graphicx}
\usepackage{dsfont}

\DeclareMathAlphabet{\mathbfit}{OML}{cmm}{b}{it}

\makeatletter
\@addtoreset{equation}{section}
\makeatother

\theoremstyle{thmstyleone}%
\newtheorem{theorem}{Theorem}
\theoremstyle{thmstyletwo}%

\newtheorem{example}{Example}%
\newtheorem{remark}{Remark}%
\newtheorem{lemma}{Lemma}[section]%
\newtheorem{corollary}{Corollary}%
\newtheorem{APPcorollary}{Corollary}[section]
\newtheorem*{claim*}{Claim}

\theoremstyle{thmstylethree}%
\newtheorem*{definition*}{Definition}

 \def \R {{\mathbb R}}
 
 \def \N {{\mathbb N}}
 \def \P {{\mathbb P}}
 
 \def \E {{\mathbb E}}

 \def \cG {\mathcal{G}}
 
 \def \cL {\mathcal{L}}

 \def \cT {\mathcal{T}}

 \def \z {{\z}}

 \def \z {{\zeta}}

\def \à {{\`{a}}}
\def \ì{{\`{\i}}}
\def \ò{{\`{o}}}
\def \è{{\`{e}}}
\def \ù{{\`{u}}}

\def \fG {\mathfrak G}

\usepackage{mathtools}
\usepackage{enumitem}

\usepackage{physics}
\usepackage{hyperref}

\def\df{\,\mathrm{d}}
\def\exp{\mathrm{e}^}

 \def \brange{{(0,2)\hspace{-0.1cm}\setminus\hspace{-0.1cm} \{1\}}}

 \usepackage{cases}

\begin{document}
\title[Double-transform Tauberian method for precise large deviations]{Double-transform Tauberian method\\  for precise large deviations}

\author{Giampaolo Cristadoro$^{1}$}
\author{Gaia Pozzoli$^{1}$}

\address{$^{1}$Dipartimento di Matematica e Applicazioni,
Universit\`a degli Studi di Milano-Bicocca,
Via R. Cozzi 55, 20125 Milano, Italy.}

\email{giampaolo.cristadoro@unimib.it}
\email{gaia.pozzoli@unimib.it}

\date{}

\begin{abstract}
In many stochastic models, the observables of interest are naturally encoded in double transforms (e.g., Laplace transforms) that couple spatial and temporal variables. Notably, the double transform often provides the only analytically tractable starting point for the study of processes with correlated increments or path constraints. We extend the Tauberian approach for precise large deviations of stochastic processes belonging to the domain of attraction of spectrally positive stable laws, previously developed for single-variable Laplace--Stieltjes transforms \cite{CP2024}, to the bivariate setting. This methodology provides a direct route to asymptotic behaviour that is otherwise difficult to characterize using single-transform techniques. As illustrative examples, we derive precise large deviations for random sums with increments correlated to the stopping time and for path-dependent observables of random walks constrained to remain positive.

    \end{abstract}

\maketitle

\section{Introduction}\label{sec:Intro}
The theory of large deviations investigates the asymptotic behaviour of probabilities of rare events and plays a central role across probability theory, statistical physics, and applied domains such as finance, queueing systems, and disordered media. Classical approaches, originating from Cram\'er's seminal work~\cite{Cramer1938}, rely on the analysis of moment generating functions and yield results at logarithmic scale under analyticity assumptions. When these assumptions fail, notably in the presence of heavy-tailed distributions, different mechanisms govern rare events, leading to precise large deviation principles typically driven by the so-called big-jump phenomenon.
 In contrast with the Cram\'er setting, however, the role played by the characteristic function of the random variable of interest does not emerge naturally from the classical techniques.

In our recent work~\cite{CP2024}, we introduced a novel approach to precise large deviations based on a uniform Tauberian theorem for families of Laplace--Stieltjes transforms. This method provides a framework for distributions in the domain of attraction of spectrally positive stable laws that reveals the role of non-analytic transforms, in analogy with the classical Cram\'er theory. The key feature of this approach is the possibility of deriving large deviation principles from a uniform control of a one-parameter family of Laplace--Stieltjes transforms near the origin.

The present work fits into the program initiated in \cite{CP2024}, where a key feature was that the dependence on the parameter indexing the family was treated separately from the transform variable. In many stochastic models, however, such a  family is not provided in an explicit form. On the contrary, the relevant quantities are often encoded directly  into transforms depending simultaneously on multiple variables, where the dependence is intrinsically coupled. Crucially, for these transforms no factorization or independent inversion is readily available. As a consequence, the analysis cannot be reduced to a single-parameter transform. This situation arises, for instance, in prototypical models with intrinsic correlations between increments and stopping mechanisms.

This lack of decoupling represents a major obstruction to the direct application of the uniform Tauberian framework developed in the one-parameter setting. Indeed, even when each marginal transform exhibits a regular asymptotic behaviour, the joint transform may contain mixed terms that cannot be controlled by separate arguments. As a consequence, the inversion with respect to one variable cannot be performed independently of the other, preventing a straightforward reduction to the results of \cite{CP2024}.

The main goal of this paper is to overcome this limitation by developing a Tauberian theory for families of double transforms. More precisely, we establish a uniform Tauberian theorem for a class of two-parameter transforms that allows us to extract asymptotic information on the underlying distributions despite the presence of coupling. This result extends the classical Tauberian paradigm and provides a tool to handle models where the interplay between space and parameter is essential.

Our approach shows that, in contrast with the single-transform case, the correct object of study is the full double transform, whose asymptotic behaviour must be properly controlled in a joint regime. This perspective is not merely technical: it reflects the intrinsic structure of several stochastic models with correlations or path-dependent features. In particular, we exhibit explicit examples where any attempt to decouple the transform fails, while the double-transform method yields precise large deviation estimates.

Among the applications, we consider randomly stopped random walks with dependent increments or path-dependent observables of random walks conditioned to stay positive.
In both these models the double transform arises naturally (for example from renewal-type equations), and encodes the full probabilistic structure of the system. Our results provide precise asymptotics for the corresponding large deviation probabilities, covering regimes that are inaccessible by previous techniques.

The paper is organized as follows. In Section~\ref{sec:UTT} we introduce the class of double transforms under consideration and state the main Tauberian theorem. Section~\ref{sec:Examples} is devoted to applications, including stopped random walks with correlated increments or path-dependent observables of random walks conditioned to stay positive.
The proofs are presented in Section~\ref{sec:proofs}.

\section{Uniform Tauberian theorem for double transforms}\label{sec:UTT}
Consider a family of Laplace--Stieltjes transforms $(\hat G_t(s))_{t\geq0}$ of real-valued functions $(G_t(x))_{t\geq0}$ locally of bounded variation on $\R$  
\[
\hat G_t(s)\coloneqq \int_{-\infty}^\infty \exp{-sx} \df G_t(x),
\]
that are assumed to exist in a right-neighborhood of the origin $s\in (0,\varepsilon)$, with $\varepsilon>0$ independent of $t$.
Let us now introduce the Laplace pair $\lambda$ of the parameter $t$ and the corresponding double transform
\[
\hat {\mathfrak G} (\lambda,s)\coloneqq \int_0^\infty \exp{-\lambda t}\hat G_t(s)\df t,
\]
that is assumed to be well-defined for $ \lambda\in(0,\infty)$. 
We  establish a Tauberian theorem for the family of double Laplace--Stieltjes transforms $(\hat {\mathfrak G} (\lambda,s))_{\lambda\in\R^+}$.
\begin{theorem}\label{thm:doubleLaplace}
Let $G_t(x)$ be non-decreasing $\forall t\geq0$ with $\lim_{x\to-\infty}G_t(x)=0$. Let $\hat G_t(s)$ be non-decreasing\footnote{The monotonicity assumption can be relaxed to $\hat G_t(s)$ non-negative and eventually non-decreasing in $t$ uniformly for $s\in(0,\varepsilon)$.} in $t$ for all $s\in(0,\varepsilon)$. Suppose that there exists $(s_\lambda)_{\lambda>0}$ with $s_\lambda>0$ for all $\lambda$ and $s_\lambda\to 0$ as $\lambda \to 0$ such that for some $\alpha>0$, $\gamma>1$ and $\ell(x)$, $L(t)$ positive slowly varying functions at infinity we have\footnote{Recall that, for any $\gamma\neq 0$, we can assume that $\ell(1/s)s^\gamma$ is monotone without loss of generality by the monotone equivalents theorem~\cite[Thm.~1.5.3]{bingham}.}
\begin{equation}\label{eq:hypLap}
\forall \theta>0,\qquad\lim_{\lambda\to0}\sup_{0<s\leq \theta s_\lambda}\left| \frac{\hat{\mathfrak G}(\lambda,s)}{L(1/\lambda)\lambda^{-\gamma}s^{-\alpha}\ell(1/s)} - 1 \right|=0.
\end{equation}
Then, defining $x_t\coloneqq 1/ s_{\frac 1 t}$,
\begin{equation}
\lim_{t\to\infty} \sup_{x\geq x_t}\left|\Gamma(\gamma)\Gamma(\alpha+1) \frac{G_t(x)}{L(t)t^{\gamma-1}\ell(x)x^\alpha}- 1 \right| =0.
\end{equation}
\end{theorem}

\begin{remark}\label{rmk:lit}
A formulation of Tauberian results for the double Laplace transform can be found in~\cite{HaanOmeyResnick1984, Diamond1987, Omey1989}. However, the problem addressed in this article requires a finer control, which is not provided by the available results in the literature. Indeed, the existing techniques\,---\,involving regular variation in dimension $2$\,---\,allow one to control only families $(\tilde x_t)_{t\geq 0}$ with $\tilde x_t \geq x_t$ that are regularly varying with respect to the parameter $t$, thereby preventing a uniform control over the whole region $x\in [x_t,\infty)$ (refer to Lemma~\ref{lem:eqPointUnif}).
\end{remark}

As a corollary, an analogue result is derived for power series: given a family  $(\hat G_n(s))_{n\in\N}$ of non-negative transforms with discrete parameter $n$, we can introduce the family of generating functions
\begin{equation}
\hat{\mathcal G}(z,s)\coloneqq \sum_{n=0}^\infty z^n \hat G_n(s),
\end{equation}
that is assumed to converge for $z\in[0,1)$.
\begin{corollary}\label{cor:Lap}
Let $G_n(x)$ be non-decreasing $\forall n\in\N$ with $\lim_{x\to-\infty}G_n(x)=0$. Let $\hat G_n(s)$ be non-decreasing in $n$ for all $s\in(0,\varepsilon)$. Suppose that there exists $(s_z)_{z \in [0,1)}$ with $s_z\to 0$ as $z\to1$ such that for some $\alpha>0$, $\gamma>1$ and $\ell(x)$ positive slowly varying function at infinity we have
\begin{equation}
\forall \theta>0,\qquad\lim_{z\to 1}\sup_{0<s\leq \theta s_z}\left| \frac{\hat{\mathcal G}(z,s)}{(1-z)^{-\gamma}s^{-\alpha}\ell(1/s)} - 1 \right|=0.
\end{equation}
Then, defining $ x_n\coloneqq 1/ s_{1-\frac 1 n}$,
\begin{equation}
\lim_{n\to\infty} \sup_{x\geq x_n}\left|\Gamma(\gamma)\Gamma(\alpha+1) \frac{G_n(x)}{n^{\gamma-1}\ell(x)x^\alpha}- 1 \right| =0.
\end{equation}
\end{corollary}

Theorem~\ref{thm:doubleLaplace} and Corollary~\ref{cor:Lap} parallel the results of~\cite[Thm.~2]{CP2024}. While this new framework encompasses several applications
 from~\cite{CP2024}, the two approaches are not equivalent. In correlated settings, the fine control of the double transforms necessary to apply the one-parameter results of~\cite{CP2024} is generally lacking; see Section~\ref{sec:Examples}.
\newline

To elucidate the mechanics of our double-transform approach in a simplified setting, we first consider the benchmark case of sums of i.i.d.\ random variables. Although this classical setup lacks the inherent coupling that motivates our general theory, it serves as a transparent toy model to show how the joint asymptotic control of the double transform recovers well-known large deviation results.
\begin{example}\label{ex:IID}
Let $(X_k)_{k\in\N}$ be an i.i.d.\ sequence of $\R^+$-valued random variables with common distribution $F_+(x)$ in the domain of attraction\footnote{By~\cite[Thm.~2.6.1]{Ibragimov}, for all $C>0$ we have $\lim_{x\to\infty}\bar F_+(Cx)/\bar F_+(x)= C^{-\beta}$, where $\bar F_+(x)\coloneqq 1-F_+(x)$, namely $\bar F_+$ is regularly varying of index $-\beta$.} of a $\beta$-stable law, with $\beta\in\brange$. Let $(\tilde X_k)_{k\in\N}$ be the corresponding possibly centred random variables
\[
\tilde X_k\coloneqq \begin{cases}
X_k\qquad &\text{if}\quad \beta \in (0,1),\\
X_k-\mu\qquad&\text{if}\quad \beta \in(1,2), 
\end{cases}\qquad k\in\N,
\]
where $\mu\coloneqq \E[X_1]$, and denote by $F(x)$ the common distribution function. 
Let $\hat F(s)$ be the Laplace--Stieltjes transform of $F(x)$.
Note that defining  $\hat F(s)\coloneqq 1+a(s)$  we have that    $a(s)=-\Gamma(1-\beta)\ell(1/s)s^\beta+o(\ell(1/s)s^\beta)$ as $s \to 0^+$ with $\ell$ positive slowly varying function at infinity. It follows that  $-1<a(s)<0$ if $\beta\in(0,1)$, while $a(s)>0$ if $\beta\in(1,2)$ and $s$ is small enough; see~\cite[Appendix~B.1]{Bianchi2022} for a brief overview.
Take the  partial sums $S_n\coloneqq \sum_{k=1}^n  \tilde X_k$, with distribution funciton $F_n(x)$.  We denote the tail distributions by $\bar{F}_n(x)\coloneqq 1-F_n(x)$. 

Consider firstly the case $\beta \in (0,1)$. We will apply Corollary~\ref{cor:Lap}  to the family $G_n(x):=\int_0^x \bar{F}_n(y)dy$.
Note  that $\hat G_n(s) = \frac{1-\hat F_n(s)}s $ is non-decreasing in $n$ for all $s\in(0,\epsilon)$, with $\epsilon>0$ arbitrarily small,  given that $\hat F_n(s)=[\hat F(s)]^n$. 

It is immediate to realize that, for $|\frac z{1-z}a(s) |<1$
\begin{eqnarray}
\hat \cG(z,s)
&\coloneqq& \sum_{n=0}^\infty z^n \hat G_n(s)=\frac1{s(1-z)}\left[ 1- \frac{1-z}{1-z\hat F (s)}\right]\nonumber\\
&=&\frac1{s(z-1)}\sum_{k=1}^\infty \left(\frac z {1-z}a(s)\right)^k.\label{eq:doubleIID}
\end{eqnarray}
Choose $s_z$ such that $a(s_z)/(1-z)\to 0$ as $z\to 1$. Note that, for  all fixed  $\theta>0$,   $\hat{\mathfrak G}(z, \tilde s_z) \to 1$ 
for any $(\tilde s_z)_{z\in[0,1)}$ such that  $\tilde s_z\leq  \theta s_z$. By means of Corollary~\ref{cor:eqPointUnif}, we thus have 
\[
\lim_{z\to 1 }\sup_{0<s\leq \theta s_z}\left| \frac{(1-z)^2\hat{\mathcal G}(z,s)}{ \Gamma(1-\beta)\ell(1/s)s^{-(1-\beta)}}-1\right|=0,\qquad\forall \,\theta>0,
\]
which satisfies the hypotheses of Corollary~\ref{cor:Lap}. By applying Lemma~\ref{lem:MDT} to $G_n(x)$, noting that we can replace $s_z$ with $s_z/(1+\eta)$ for any $\eta>0$, we immediately conclude that
\begin{equation}\label{eq:LDforIID}
\lim_{n\to\infty}\sup_{x\geq x_n}\left| \frac{\bar F_n(x)}{n\ell(x)x^{-\beta}}-1\right|=\lim_{n\to\infty}\sup_{x\geq x_n}\left| \frac{\bar F_n(x)}{n\bar F(x)}-1\right|=0,
\end{equation}
with $ x_n$ such that $n \ell( x_n) x_n^{-\beta}\to 0$ as $n\to\infty$.\footnote{ Recall that this large deviation condition is optimal in the sense that it is at the boundary with the ``natural'' fluctuations of the GCLT.} 

The  large deviation result~\eqref{eq:LDforIID} is also  valid  for $\beta\in(1,2)$ and can be derived analogously with    $\hat G_n(s) = \frac{\hat F_n(s)-1}{s^2} $  and applying Lemma~\ref{lem:MDT} twice.
\end{example}

The same result was obtained in~\cite[Cor.~1]{CP2024} but under a control corresponding to the termwise inversion of the coefficients of $[a(s)]^k$ in~\eqref{eq:doubleIID}, yielding a full asymptotic expansion factorized in $n$ and $s$.

More generally, observe that the control required on the subleading terms in the asymptotic expansion of $\hat{\mathfrak G}(\lambda,s)$ in Theorem~\ref{thm:doubleLaplace} is weaker than requiring control of its expansion as a power series in $s$. The latter would allow one to invert, term by term in $\lambda$, the coefficients and thereby recover uniform control with respect to the parameter $t$.

\section{Applications}\label{sec:Examples}
In this section, we illustrate the practical utility and versatility of the uniform Tauberian framework established in Section~\ref{sec:UTT} by applying it to stochastic models with different underlying structures. Firstly, focusing on first-passage problems and stopped random walks, we show how the double-transform approach naturally emerges from renewal-type equations to encode the full probabilistic structure of the system. Indeed, it captures the intrinsically coupled dynamics, allowing us to derive precise large deviation asymptotics in correlated regimes where traditional decoupling techniques fail.
Secondly, we show that even when maintaining i.i.d.\ increments for the random walk as in Example~\ref{ex:IID}, conditioning on a path-dependent event naturally requires the use of a double transform.

To this end, the following notation proves useful.
Consider the joint control-cost process  $(S,C)\coloneqq (S_n,C_n)_{n\in\N_0}$ defined by a sequence of i.i.d.\ vector-valued increments $(X_k,Y_k)_{k\in\N}\in (\R\times [0,\infty))^{\N}$. The process evolves according to 
\begin{equation}\label{eq:CCprocess}
S_0\coloneqq 0, \quad C_0\coloneqq 0\quad\text{and}\quad S_n\coloneqq\sum_{k=1}^n X_k, \quad C_n\coloneqq \sum_{k=1}^n Y_k \quad\text{for}\quad n\in\N.
\end{equation}
As anticipated, we will consider a non-negative stochastic process $(Z_t)_{t\geq 0}$ constructed using a positive cost $C_n$ in two different ways: in the first set of examples $Z_t$ will denote the cost $C_{\tau(t)}$ (possibly up to a deterministic shift) accumulated  up to a stopping time $\tau(t)$; in the second set $Z_{t}$ will correspond to the cost accumulated while the control variable S is conditioned to stay positive.  Let  $R_t(x)$ denote the distribution function on $[0,\infty)$ associated with the random variable $Z_t$ for each $t\geq 0$ and let $\bar R_t(x)\coloneqq 1-R_t(x)$ be the corresponding tail distribution. We easily observe that \footnote{Note that $G_t(x)=\int_0^x \bar R_t(y)\df y\leq x$, hence the corresponding double transform is well defined $\hat{\mathfrak G}(\lambda, s)\leq (s\lambda)^{-1}<\infty$ for all $s,\lambda \in(0,\infty)$.}
\[
\hat G_t(s)\coloneqq\frac{\hat R_t(0)-\hat R_t(s)}s=\int_0^\infty\exp{-sx}\bar R_t(x)\df x, \qquad s\in(0,\infty),
\]
and $\hat G_t(s)$ is non-decreasing in $t$ as $\bar R_t(x)$ is non-decreasing in $t$ for all $x\in \R^+$.

\subsection{Correlated random sums via renewal-type equations}\label{sec:leapGut}
Fix a threshold level $t\geq 0$, and define the first-ladder epoch 
\begin{equation}\label{eq:FPT}
\tau(t)\coloneqq \inf\{ n>0\,:\, S_n>t\},
\end{equation}
i.e.\ the first-passage time to~$(t,\infty)$ of the control process $S$.
Our aim is to characterize the large deviation probabilities of the cost variable stopped at time $\tau(t)$ i.e. $Z_t:=C_{\tau(t)}$.
In the remainder, all processes are assumed to admit a probability density function.
Let $h_t(c)$ be the probability density function of $Z_t$. 
Furthermore, let $p(x,y)$ denote the joint density of the increments $(X_k,Y_k)_{k\in\N}\in ([0,\infty)\times[0,\infty))^{\N}$. Define $p_Y(x)\coloneqq \int_0^\infty p(x,y)\df x$ as the common marginal density function of the random variables $(Y_k)_{k\in\N}$\,---\,analogously for $p_X(x)$. 
\newline

We now show that  $h_t(c)$  satisfies a renewal-type equation and this will let us write its double transform explicitly. Start by writing the backward equation
\begin{equation}\label{eq:recNotTransf}
h_t(c)=\int_0^t \df x\int_0^\infty \df y\,h_{t-x}(c-y) p(x,y)+\int_t^\infty p(x,c)\df x.
\end{equation}
Intuitively, the first term represents the contribution from trajectories in which a first jump occurs below the threshold $t$, with a corresponding cost $y$; the remaining jumps (collectively constrained to sum up to $t-x$) then fill the gap to reach the total cost $c$. The second term accounts for the complementary case, where the first jump exceeds $t$ and the process is halted at $c$. 
We introduce the Laplace transform $\hat h_t(s)\coloneqq \int_{0}^\infty \exp{-s c}h_t(c)\df c$ with respect to $c$, and by the convolution theorem we get
\begin{align}\label{eq:recCost}
\hat h_t(s)
&=\int_0^t \hat h_{t-x}(s) \hat p(x; s)\df x+\int_t^\infty \hat p(x;s)\df x,
\end{align}
where
\begin{equation}\label{eq:replaced}
\hat p(x; s) \coloneqq \int_{0}^\infty p(x,y)\exp{-sy}\df y.
\end{equation}
If we now take the Laplace transform $\hat{\mathfrak h}(\lambda,s) \coloneqq  \int_0^\infty \exp{-\lambda t}\hat h_t(s)\df t$ with respect to $t$, and denote by $\hat p(\lambda,s)\coloneqq \int_0^\infty \exp{-\lambda x}\hat p(x;s)\df x$ the double Laplace transform of the joint density function $p(x,y)$, we obtain\footnote{Note that $\hat p(0,s)=\hat p_Y(s) $ by Tonelli's theorem.} 
\begin{align}\label{eq:doubleCostF}
\hat{\mathfrak h}(\lambda,s)
& =\frac 1 \lambda\frac{ \hat p(0,s)-\hat p(\lambda,s)}{1-\hat p (\lambda, s)}
= \frac 1 \lambda\left[ 1- \frac{1-\hat p_Y(s)}{1-\hat p(\lambda,s)} \right].
\end{align}
\newline

In what follows, we leverage the explicit expression of the double transform to derive precise large deviation estimates using  Theorem \ref{thm:doubleLaplace}.
More explicitly, we will apply Theorem~\ref{thm:doubleLaplace} to 
\begin{align}\label{eq:doubleCost}
\hat {\mathfrak G} (\lambda,s)&\coloneqq \frac{\lambda^{-1}-\hat{\mathfrak h}(\lambda,s)}s
= \frac 1 {s\lambda} \frac{1-\hat p_Y(s)}
{1-\hat p(\lambda,s)}.
\end{align}
We will then rely on Lemma~\ref{lem:MDT} in order to transfer the result from $G_t(x)$ to $\bar R_t(x)=\int_x^\infty h_t(c) \df c$.
The control on the asymptotic behaviour of $\hat{\mathfrak G} (\lambda,s)$
will allow us to obtain the large deviation probabilities of the family $(C_{\tau(t)})_{t\geq 0}$. 
\newline

Below, we instantiate the cost variable for three distinct models and leverage Theorem~\ref{thm:doubleLaplace} to derive their corresponding large deviation principles. In the first instance, the cost $C$ is chosen to be independent of the control process $S$. This configuration represents a random sum with independent increments and provides a fundamental benchmark; it shows that the double transform derived from a renewal-type equation recovers, in a direct way, the corresponding result in~\cite[\S~3.2]{CP2024} as a transparent special case.
 
Conversely, the second example considers a cost $C$ that is inherently coupled with the control process $S$. This case specifically addresses the asymptotic behavior of the leapover (overshoot) above a threshold $t$. Our findings contribute to the characterization of precise large deviations for this class of processes, which have remained, to our knowledge, largely unexplored in this specific regime.
Finally, we consider a third model which combines the previous two settings to obtain large deviation estimates for a process inspired by the work of Gut and Janson~\cite{GJ1983}, providing new results that extend the current understanding of this class of models.

\subsubsection{Random sums}\label{sec:RS}
Take a control-cost process~\eqref{eq:CCprocess} in which the increments $X_k$ and $Y_k$ are independent. In this case, the joint probability density function factorizes as $p(x,y)=p_X(x) p_Y(y)$, and~\eqref{eq:doubleCost} provides
\begin{align}\label{eq:doubleRS}
\hat {\mathfrak G} (\lambda,s)
& = \frac 1 {s\lambda} \frac{1-\hat p_Y(s)}{1-\hat p_X(\lambda) \hat p_Y(s)}.
\end{align}
\begin{example} \label{ex:RS}
Assume that the increments $Y_k$ of the process under large deviation analysis belong to the domain of attraction of a $\beta$-stable law as in Example~\ref{ex:IID}, namely that $\hat p_Y(s)=1+a(s)$ with $a(s)=-\Gamma(1-\beta)\ell(1/s)s^\beta+o(\ell(1/s)s^\beta)$ as $s\to 0^+$ and $\beta\in (0,1)$, and let $X_k$ be such that\footnote{ Throughout this paper, given two functions $f(x)$ and $g(x)$ we write $f(x)\sim g(x)$ as $x\to 0$ if and only if $\lim_{x\to 0} f(x)/g(x) = 1$.} $1-\hat p_X(\lambda)\sim c_X \lambda ^\gamma$ as $\lambda\to 0$ for some $c_X>0$, $\gamma\in(0,1]$.
For $|\frac{\hat p_X(\lambda)}{1-\hat p_X(\lambda)} a(s)|<1$ we can write
\begin{align}\label{eq:RSbeta}
\hat {\mathfrak G} (\lambda,s)
&= -\frac 1 {s\lambda} \sum_{k=1}^\infty \frac 1 {\hat p_X(\lambda)}\left( \frac{\hat p_X(\lambda)}{1-\hat p_X(\lambda)} a(s)\right)^k.
\end{align}
Note that for the degenerate case of deterministic sums we recover an expression closely related to~\eqref{eq:doubleIID} of Example~\ref{ex:IID}; see  Appendix~\ref{app:RS} for a detailed  comparison.
Moreover, choosing a scaling $s_\lambda$ such that $\ell(1/s_\lambda)s_\lambda^\beta/\lambda^\gamma\to 0$ as $\lambda \to 0$, we establish 
\[
\lim_{\lambda \to 0}\sup_{0<s\leq \theta s_\lambda}\left| \frac{c_X \lambda^{\gamma+1}\hat {\mathfrak G} (\lambda,s)}{ \Gamma(1-\beta)\ell(1/s)s^{-(1-\beta)}}-1\right|=0,\qquad\forall \,\theta>0,
\]
and we conclude that, for every $(x_t)_{t\geq 0}$ such that $t^\gamma \ell(x_t) x_t^{-\beta} \to 0$ as $ t \to \infty$,
\[
\lim_{t \to \infty}\sup_{x\geq x_t}\left| \frac{c_X\Gamma(\gamma)\bar R_t(x)}{t^\gamma \ell(x) x^{-\beta}} -1 \right|=0.
\]
\end{example}

It should be noted that the independence of the increments enabled the explicit factorization shown in~\eqref{eq:RSbeta}, which in turn allows for term-by-term inversion with respect to $\lambda$.  Such a decoupling  mirrors the time-domain analysis performed on the one-parameter family $(C_{\tau(t)})_{t\geq 0}$ in~\cite[\S~3.2.2]{CP2024}. This situation is similar to that of Example ~\ref{ex:IID}  -- further mathematical details regarding this comparison are deferred to Appendix~\ref{app:RS}.

Having revisited this baseline scenario, we now turn to illustrate the full power of Theorem~\ref{thm:doubleLaplace} through genuinely correlated models that cannot be traced back to the single-transform framework in~\cite{CP2024}, as detailed in Corollaries~\ref{cor:Leap} and~\ref{cor:Gut}.

\subsubsection{Leapover}\label{sec:leap}
Let us consider the cost process with increments $Y_k= X_k$ and assume that $X_k$ are in the domain of attraction~\cite[Thm.~2.6.1 and Thm.~2.6.5]{Ibragimov} of a $\beta$-stable law with $0<\beta<1$. The corresponding control process $(S_n)_{n\in\N_0}$, that is a walker covering long distances in a single instantaneous jump, is called a generalized \emph{L\'evy flight}~\cite{Mandelbrot1982,SK1986}. The random variable $L_t\coloneqq C_{\tau(t)}-t$, also known in the literature as \emph{leapover}~\cite{Eliazar2004, Koren2007b, Koren2007}, represents the magnitude of the overshoot beyond the barrier at the first-passage event. Alternatively, $L_t$ can be interpreted as the first waiting time, after the onset of the observation, in a renewal process with a broad waiting time distribution~\cite{GL2001,Barkai2003,Margolin2006}.

\begin{corollary}\label{cor:Leap}
Let $(S_n,C_n)_{n\in\N_0}$ be as in Section~\ref{sec:leapGut}, with $Y_1=X_1$ and $\P(X_1>x)\sim \ell(x) x^{-\beta}/\Gamma(1-\beta)$ as $x\to\infty$ for some $\beta\in(0,1)$, $\ell$ positive slowly varying function at infinity. Then, for $(x_t)_{t\geq 0}$ such that $x_t\to\infty $ and $t/ x_t\to 0$ as $t\to\infty$, we have
\begin{equation*}
\lim_{t\to\infty} \sup_{x\geq  x_t}\left|\Gamma(1-\beta)\Gamma(1+\beta) \frac{\ell(t)\P(L_t>x)}{t^{\beta}x^{-\beta}\ell(x)}- 1 \right| =0.
\end{equation*}
\end{corollary}

\begin{proof}
With the choice $Y_k= X_k$, we have to consider $p(x,y)=p_X(x)\delta(y-x)$ and~\eqref{eq:replaced} gives $\hat p(x;s)= p_X(x)\exp{-sx}$, $\hat p(\lambda, s) = \hat p_X(\lambda+s)$, $\hat p_Y(s)=\hat p_X(s)$.
Hence, from~\eqref{eq:doubleCost} we obtain 
\begin{align*}
\hat {\mathfrak G} (\lambda,s)
& =\frac 1 {s\lambda} \frac{1-\hat p_X(s)}{1-\hat p_X(s+\lambda)}.
\end{align*}
By assumption (see~\cite[Thm.~2.6.5]{Ibragimov}), $\hat p_X(s)=1+a(s)$ with $a(s)=-\ell(1/s)s^\beta+o(\ell(1/s)s^\beta)$ as $s\to 0$.
Choose $(s_\lambda)_{\lambda> 0}$ such that $s_\lambda/\lambda\to 0$ as $\lambda \to 0$.  By means of Corollary~\ref{cor:eqPointUnif},  it is sufficient  to control, for  all fixed  $\theta>0$,  the behaviour of $\hat{\mathfrak G}(\lambda, \tilde s_\lambda)$ 
for any $(\tilde s_\lambda)_{\lambda>0}$ such that  $\tilde s_\lambda\leq  \theta s_\lambda$ for all $\lambda$. 
Recalling the properties of the little-o notation, we can write
\begin{align*}
\hat{\mathfrak G}(\lambda, \tilde s_\lambda)
&=\frac 1 {\tilde s_\lambda\lambda}
\frac{\tilde s_\lambda^\beta\ell(1/\tilde s_\lambda)\left(1+\frac{o(\tilde s_\lambda^\beta\ell(1/\tilde s_\lambda))}{\tilde s_\lambda^\beta\ell(1/\tilde s_\lambda)}\right)} 
{\lambda^\beta\left(1+\frac{\tilde s_\lambda}\lambda\right)^\beta \ell\left(\frac 1{\lambda +\tilde s_\lambda}\right) \left(1+\frac{o((\lambda +\tilde s_\lambda)^\beta \ell( 1/ (\lambda+\tilde s_\lambda)) )}{(\lambda +\tilde s_\lambda)^\beta \ell( 1/ (\lambda+\tilde s_\lambda))}\right)}\\
&=\frac {\tilde s_\lambda^{\beta-1}}{\lambda^{\beta+1}}\frac{\ell(1/\tilde s_\lambda)}{\ell(1/\lambda)}+o\left(\frac {\tilde s_\lambda^{\beta-1}}{\lambda^{\beta+1}} \frac{\ell(1/\tilde s_\lambda)}{\ell(1/\lambda)}\right),\qquad \text{as}\quad \lambda \to 0,
\end{align*}
given that by the Uniform Convergence Theorem~\cite[Thm.~1.2.1]{bingham} we have\footnote{Looking at $\ell\left( (1+\frac{\tilde s_\lambda}\lambda)^{-1}/\lambda\right)$ we have a moving multiplier but, for $\lambda$ small enough, $\left(1+\frac{\tilde s_\lambda}\lambda\right)^{-1}\in[1/2,3/2]$. The claim then follows since
\[
\lim_{x\to\infty}\sup_{c\in[1/2,3/2]}\left|\frac{\ell(cx)}{\ell(x)}-1\right|=0.
\]}
\[
\lim_{\lambda \to 0} \frac{\ell\left(\frac 1{\lambda +\tilde s_\lambda}\right)}{\ell(1/ \lambda)}=1.
\]

Recalling that $\hat{\mathfrak G}(\lambda, s) = \frac{\lambda^{-1}-\hat{\mathfrak h}(\lambda,s)}s$ we  get
\begin{equation}\label{eq:Cleap}
\lim_{\lambda\to0}\sup_{0<s\leq \theta s_\lambda}\left|\ell(1/\lambda) \frac{\lambda^{-1}-\hat{\mathfrak h}(\lambda,s)}{\lambda^{-\beta-1}s^\beta\ell(1/s)} - 1 \right|=0\quad\text{for all}\quad \theta>0.
\end{equation}
Now, noting that by definition $\mathcal L_{L_t}(s)\coloneqq \E(\exp{-sL_t})=\exp{st}\hat h_t(s)$ and as a consequence
\[
\mathcal L_t\{\mathcal L_{L_t}(s)\}(\lambda)\coloneqq \int_0^\infty \exp{-(\lambda-s) t}\hat h_t(s) \df t=\hat {\mathfrak h}(\lambda-s, s),
\]
by~\eqref{eq:Cleap}, Theorem~\ref{thm:doubleLaplace}\,---\,where the choice of the family $(s_\lambda)_{\lambda>0}$ is defined only up to multiplication by a positive constant by assumption~\eqref{eq:hypLap}\,---\,and Lemma~\ref{lem:MDT}, we can conclude that 
\begin{equation*}
\lim_{t\to\infty} \sup_{x\geq  x_t}\left|\Gamma(1-\beta)\Gamma(1+\beta) \frac{\ell(t)\P(L_t>x)}{t^{\beta}x^{-\beta}\ell(x)}- 1 \right| =0,
\end{equation*}
with $x_t\coloneqq 1/ s_{\frac 1 t}$.
\end{proof}

Note that the large deviation condition  $t/ x_t\to 0$ as $t\to\infty$ is sharp, as it lies precisely at the boundary of the support constraint $h_t(c)=0$ for $c\leq t$ inherent to the model's construction; see Appendix~\ref{app:BoundedRegion} for a detailed discussion.

\medskip
The proof reveals that, unlike the decoupled structures seen in Examples~\ref{ex:IID} and~\ref{ex:RS}, the joint transform $\hat{\mathfrak h}(\lambda, s)$ does not generally allow for an explicit form as a series of functions factorized in $s$ and $\lambda$\,---\,except for specific choices of the distribution of the increments; see Appendix~\ref{app:ML} for an example. This observation is the fundamental reason for requiring direct asymptotic control over the double transform.

\subsubsection{Stopped two-dimensional random walks}\label{sec:OSGut}
Inspired by~\cite{Kaijser1971,GJ1983} and~\cite[Ch.~4]{Gut2009}, let us assume that the cost process~\eqref{eq:CCprocess} can be expressed as the sum of two components: one correlated with the control process, as in Section~\ref{sec:leap}, and another independent of it, as in Example~\ref{ex:RS}.  More precisely, $Y_k\coloneqq X_k+W_k$, where $(W_k)_{k\in\N}\in\R_+^\N$ is independent of $(X_k)_{k\in\N}\in\R_+^\N$.

For comparison with the leapover case, the variable $C_{\tau(t)}$ may now be interpreted as a quantitative measure of longitudinal dispersion by means of the distribution of the total time elapsed  until the control variable exceeds the threshold $t$. Indeed, the two components of the cost process can be viewed as corresponding to a mobile phase (with constant longitudinal velocity) and a stationary phase of the walker, respectively.

By way of illustration, and without loss of generality, suppose that the mobile phase is governed by a finite mean distribution (e.g., Poisson exponential law), and that the waiting times in the immobile zone follow a power law distribution~\cite{SBMB2003}. Specifically, for some $\mu>0$, $\ell$ positive slowly varying function at infinity and $\beta\in(0,1)$
\begin{equation}\label{eq:GutDistr}
\hat p_X(s)=1-\mu s+o(s),\qquad \hat p_W(s)=1-\ell(1/s)s^\beta+o(s^\beta).
\end{equation}
\begin{corollary}\label{cor:Gut}
Let $(S_n,C_n)_{n\in\N_0}$ be as in Section~\ref{sec:leapGut}, with $Y_1=X_1+W_1$ defined in~\eqref{eq:GutDistr}. Then, for $( x_t)_{t\geq 0}$ such that $ x_t\to\infty $ and $t/ x_t^\beta\to 0$ as $t\to\infty$, we have
\begin{equation*}
\lim_{t\to\infty} \sup_{x\geq x_t}\left|\mu\Gamma(1-\beta)\ell(t) \frac{\P(C_{\tau(t)}>x)}{t\ell(x)x^{-\beta}}- 1 \right| =0.
\end{equation*}
\end{corollary}

\begin{proof}
We write $p(x,y)=p_X(x)p_W(y-x)$, where $p_W(w)$ denotes the common density distribution of the i.i.d.\ random variables $(W_k)_{k\in\N}$.
Hence~\eqref{eq:replaced} provides $\hat p(x;s)=p_X(x)\int_0^\infty p_W(y-x)\exp{-sy}\df y$, $\hat p_Y(s)=\hat p_X(s) \hat p_W(s)$ and $\hat p(\lambda, s)=\hat p_X(s+\lambda) p_W(s)$, and the double Laplace trasform of the stopped cost process $C_{\tau(t)}$ gives
\[
\hat {\mathfrak G} (\lambda,s) =\frac 1 {s\lambda}\frac{1-\hat p_W(s)\hat p_X(s)}{1-\hat p_W(s)\hat p_X(s+\lambda)}.
\]
Proceeding as in the proof of Corollary~\ref{cor:Leap}, given $(s_\lambda)_{\lambda> 0}$ such that $s_\lambda^\beta/\lambda\to 0$ as $\lambda \to 0$, note that for every $\theta>0$ and any $(\tilde s_\lambda)_{\lambda>0}$ such that $\tilde s_\lambda\leq \theta s_\lambda$ for all $\lambda$ we get
\[
\hat {\mathfrak G} (\lambda,\tilde s_\lambda) = \frac c\mu \frac {\tilde s_\lambda^{\beta-1}}{\lambda^2}\frac{\ell(1/\tilde s_\lambda)}{\ell(1/\lambda)}+o\left(\frac {\tilde s_\lambda^{\beta-1}}{\lambda^2}\frac{\ell(1/\tilde s_\lambda)}{\ell(1/\lambda)}\right) \quad\text{as}\quad \lambda \to 0.
\]
As a consequence of Corollary~\ref{cor:eqPointUnif}, Theorem~\ref{thm:doubleLaplace} and Lemma~\ref{lem:MDT}, defining $ x_t\coloneqq 1/ s_{\frac 1 t}$, this leads us to the conclusion
\begin{equation*}
\lim_{t\to\infty} \sup_{x\geq x_t}\left|\mu\Gamma(1-\beta)\ell(t) \frac{\P(C_{\tau(t)}>x)}{t\ell(x)x^{-\beta}}- 1 \right| =0.
\end{equation*}
\end{proof}
As noted in Appendix~\ref{app:ML} for the leapover case,  direct control of the double transform is crucial in the presence of correlations. Exception occurs only for specific distributions (e.g  for exponential and Mittag--Leffler) where the error term is explicit and the coupled contributions in $s$ and $\lambda$ can be inverted exactly with respect to $\lambda$, as shown in Example~\ref{ex:ML}. However, without full knowledge of the distribution, as in~\eqref{eq:GutDistr}, explicit control of the error term in $t$ is generally unavailable.

\subsection{Path-observables of random walks conditioned to stay positive }\label{sec:CSP}
Consider the control-cost process~\eqref{eq:CCprocess} with $(X_k,Y_k)_{k\in\N}\in (\R\times[0,\infty))^{\N}$.
Our aim is to characterize the large deviation probabilities of the cost variable when the control variable is conditioned to stay positive. To this aim, let 
$(\cT_k,H_k)_{k\in\N_0}$ denote the bivariate renewal process of \emph{strict} ladder times and heights, defined by
\[
\cT_0=0,\quad \cT_k\coloneqq \inf\{j>\cT_{k-1}\,:\,S_j>S_{\cT_{k-1}}\}\quad\text{for}\quad k\in\N,
\]
with the convention $\inf\{\emptyset\}=+\infty$, and with $H_k\coloneqq S_{\cT_k}$. Moreover, let $\cT^-\coloneqq \inf\{n\geq1\,:\,S_n\leq 0\}$ denote the first \emph{weak} decreasing ladder time.
We define  $\{Z_n > x\}:= \{C_n>x, \cT^- >n\}$ and write $g_C$ for the Green's function (i.e., renewal mass function) in the bivariate renewal process  --- see also~\cite[Eq.~(7)]{AD1999}\cite{D2012,DJ2012} --- with the presence of a cost, defined as
\begin{equation}
g_C(n,\df x, \df y)\coloneqq \sum_{k=0}^n \P(\cT_k=n,H_k\in \df x,C_n\in \df y),\quad n\geq 0,\quad x,y\geq 0.
\end{equation}
As a fundamental observation, $g_C(n,\df x, \df y)$ corresponds to the distribution of the control-cost process $(S,C)$ while S is conditioned to stay positive
\[
g_C(n,\df x,\df y) =\P(S_n\in\df x , Z_n\in\df y),
\]
as an immediate consequence of the dual relations~\cite[\S~XII.2]{Feller2}; see also~\cite[Lem.~2.1 and Cor.~2.2]{AD2001}.\footnote{For fixed $n$, introduce new variables $X_1^*=X_n,\dots,X_n^*=X_1$, and denote the corresponding partial sums by $S_k^*=S_n-S_{n-k}$ where $k=0,\dots,n$. Given that
\[
g(n,\df x)=\P(S_n>S_j\text{ for } j=0,\dots, n-1 \text{ and } S_n\in\df x),
\]
it is not difficult to realize that if $S_n>S_j$ for $j=0,\dots, n-1$ and $S_n\in\df x$, then $S^*_1>0,\dots,S^*_n>0$ and $S^*_n(=S_n)\in\df x$. This proves the validity of~\cite[Eq.~(19)]{D2012}. It is immediate that considering jointly $S_n\in\df x$ and $C_n\in\df y$ does not invalidate the proof (in fact, $C_n^*=C_n$). }
Let
\[
\hat g_C(n;u,s)\coloneqq \int_0^\infty \df x \,\exp{-ux}\int_0^\infty \df y\, \exp{-sy} g_C(n,\df x,\df y),\quad u,s\geq 0,
\]
denote the double Laplace transform of $g_C(n,\df x,\df y)$ with respect to the control-cost process. If $n=0$ we trivially get $\hat g_C(0;u,s)=1$. So let us now assume that $n\geq 1$.
From a slight generalization of~\cite[Eq.~(13)]{AD1999}, we know that for $1\leq k\leq n$
\[
\P(\cT_k=n,H_k=x, C_n=y)=\frac k n \P({\mathrm N}_x=k,S_n=x, C_n=y), \quad x,y \geq 0, 
\]
where ${\mathrm N}_x\coloneqq \inf\{j\,:\, H_j\geq x\}$.
Hence for $n\geq 1$ we get 
\begin{align*}
\hat g_C(n;u,s)
& = \frac 1 n  \int_0^\infty \df x\,\exp{-ux}\int_0^\infty \df y \,\exp{-sy} \sum_{j=1}^n j \, \P({\mathrm N}_x=j,S_n=x,C_n=y)\\
& = \frac 1 n \E({\mathrm N}_{S_n}\exp{-uS_n}\exp{-sC_n};S_n>0).
\end{align*}
Introducing the generating function with respect to $n$, and by dominated convergence\footnote{Since by definition $0\leq\mathrm{N}_{S_n}\leq n$ and, denoting $\mathfrak F_n(\alpha)\coloneqq \E(\alpha^{\mathrm{N}_{S_n}}\exp{-uS_n}\exp{-sC_n};S_n>0)$, for any fixed $n$ and any $\alpha \in (0,1]$
\[
\left|\frac{\mathfrak F_n(\alpha)-\mathfrak F_n(1)}{\alpha-1}\right|= \E\left(\frac{\alpha^{{\mathrm N}_{S_n}}-1}{\alpha-1}\exp{-uS_n}\exp{-uC_n};S_n>0\right)\leq \E({\mathrm N}_{S_n};S_n>0),
\]
the difference quotients are dominated by $n$ and the Dominated Convergence Theorem justifies $\mathfrak F'_n(1)=\E({\mathrm N}_{S_n}\exp{-uS_n}\exp{-sC_n};S_n>0)$.}, we obtain that for any $z\in[0,1)$
\begin{align*}
\hat g_C(z,u,s)
& \coloneqq \sum_{n=0}^\infty z^n \hat g_C(n;u,s)\\
& = 1+ \sum_{n=1}^\infty \frac {z^n} n  \left[\frac{\partial}{\partial \alpha}\E(\alpha^{{\mathrm N}_{S_n}}\exp{-uS_n}\exp{-sC_n};S_n>0) \right]_{\alpha=1}.
\end{align*}
Exchanging the sum and the derivative\footnote{A sufficient condition to justify this (by dominated convergence) is the absolute summability of the series of derivatives, guaranteed by the fact that $|z|\E(\exp{-uX_1})<1$.}, we can use the generalization of~\cite[Eq.~(14)]{AD1999}\footnote{Observe that for $\alpha=1$ we recover the generalized Spitzer-Baxter identity~\cite[Thm.~2.1]{Bianchi2022}.} to write
\begin{align}
\hat g_C(z,u,s) & = 1+ \frac{\partial}{\partial \alpha} \left[ - \log (1-\alpha\E[z^{\cT_1}\exp{-uH_1}\exp{-sC_{\cT_1}}])\right]_{\alpha=1}.\label{eq:consAlili}
\end{align}
Note that $\hat R_n(s)=\hat g_C(n,0,s)$, hence
\(
\hat R(z,s)\coloneqq\sum_{n=0}^\infty z^n \hat R_n(s)=\hat g_C(z,0,s).
\)
To obtain precise large deviation estimates, we will invoke Corollary~\ref{cor:Lap} applied to
\[
\hat{\mathcal G}(z,s) \coloneqq \frac {\hat R(z,0) - \hat R(z,s)} s=\frac {\hat g_C(z,0,0) - \hat g_C(z,0,s)} s,
\]
after specifying the cost $C$.

\subsubsection{Total path-length of a random walk conditioned to stay positive}
Let us consider the cost process defined by $Y_k\coloneqq |X_k|$, with $X_k$'s  in the domain of attraction of a $\beta$-stable law with $0<\beta<1$. 
The random variable $C_n$ represents the total distance travelled by the generalized L\'evy flight $(S_n)_{n\in\N_0}$ after $n$ steps. Hereinfater, we will assume that $(X_n)_{k\in\N}$ has an absolutely continuous distribution for simplicity of exposition.
\begin{corollary}\label{cor:Cond}
Let $(S_n,C_n)_{n\in\N_0}$ be as in~\eqref{eq:CCprocess}, with $Y_1=|X_1|$ and $\P(X_1<-x)=\P(X_1>x)\sim \ell(x) x^{-\beta}/\Gamma(1-\beta)$ as $x\to\infty$ for some $\ell(x)$ positive slowly varying function, $\beta\in(0,1)$. Then, for $( x_n)_{n\in\N}$ such that $ x_n\to\infty $ and $n/ x_n^\beta\to 0$ as $n\to\infty$, we have
\begin{equation}\label{eq:TotalLengthConditioned}
\lim_{n\to\infty}\sup_{x\geq x_n}\left| \frac{\P(C_n>x,\cT^->n)}{2n\P(\cT^{-}>n)\P(X_1>x)}-1\right|=0.
\end{equation}
\end{corollary}
\begin{proof}
Observe that $(S,C)$ is symmetric, in the sense that for all $x,y\in \R$ and $n\in\N$
\[
\P(S_n\in\df x, C_n\in\df y)=\P(-S_n\in\df x, C_n\in\df y).
\]
Setting $u=0$ in~\eqref{eq:consAlili}, by~\cite[Cor.~3.4]{Bianchi2022} we get the closed-form expression
\[
\hat g_C(z,0,s)=\frac 1{\sqrt{1-z\hat p_Y(s)}}.
\]

Let us compute
\begin{eqnarray*}
\hat{\mathcal G}(z,s) &=& \sum_{n=0}^\infty z^n\int_0^\infty\exp{-sx}\P(C_n>x,\cT^->n) \df x,\\
&=&\frac1s\left[   \frac1{\sqrt{1-z}}-\frac1{\sqrt{1-zp_Y(s)}} \right].
\end{eqnarray*}
For $\hat p_Y(s)=1+a_Y(s)$, choosing  $\tfrac z {1-z}|a_Y(s)|<1$, we can expand
\begin{align*}
\hat{\mathcal G}(z,s) 
&= -\frac 1 {s\sqrt{1-z}} \left[\frac z {1-z} \frac{a_Y(s)}2+\sum_{k=2}^\infty \frac{(2k)!}{(2^k k!)^2} \left(\frac z {1-z}a_Y(s)\right)^k \right].
\end{align*}
We can therefore apply Corollary~\ref{cor:Lap}, with $(s_z)_{z \in [0,1)}$ such that $s_z\to 0$ and $ a_Y(s_z)/(1-z)\to 0$ as $z\to1$ and, by Lemma~\ref{lem:MDT}, we obtain that

\[
\lim_{n\to\infty}\sup_{x\geq x_n}\left| \frac{\sqrt{\pi}\P(C_n>x,\cT^->n)}{\sqrt{n}\P(Y_1>x)}-1\right|= 0,
\]
where $( x_n)_{n\in \N}$ is such that $ x_n\to \infty$ and $n/ \ell(x_n)x_n^\beta$ as $n\to\infty$.
Noting that $P(Y_1>x)=2P(X_1>x)$ and that\footnote{This is a well-known result in the literature: by the Sparre Andersen identity~\cite[\S{XII.7}]{Feller2} --- for absolutely continuous distribution the contribution from $S_n=0$ is irrelevant --- $\hat g_{\cT^-}(z)\coloneqq \sum_{n=1}^\infty \P(\cT^-=n)z^n=1-\sqrt{1-z}$. Hence, 
\[
\sum_{n=0}^\infty \P(\cT^->n)z^n=\frac{1-\hat g_{\cT^-}(z)}{1-z}=\frac 1 {\sqrt{1-z}}=\sum_{j=0}^\infty \binom{\tfrac 1 2 +n-1}{n} z^n ,\qquad z\in(0,1),
\]
which implies $\P(\cT^->n)=  (2n)!/(2^n n!)^2\sim 1/\sqrt{\pi n}$ as $n \to \infty$.} $\P(\cT^{-}>n)\sim 1/\sqrt{\pi n}$ yields the result~\eqref{eq:TotalLengthConditioned}.
\end{proof}
While the factorized form of $\hat{\mathcal G}(z,s) $ suggests a possible reduction to a single-parameter Tauberian theorem, such a structure is not directly accessible starting from the family in the time domain, in contrast to Examples~\ref{ex:IID} and~\ref{ex:RS}. 
Here, the relevant quantities are attainable only through recurrence relations that yield the joint transform; the double-transform framework is therefore not merely optional, but necessary. We employ it throughout the analysis because it provides a more natural and direct route to the asymptotic results than a post-hoc single-variable inversion.

\section{Proofs}\label{sec:proofs}
\subsection{Proof of Theorem~\ref{thm:doubleLaplace}}
\begin{proof}
We introduce the following notation
\[
\hat {\mathfrak G} (\lambda,s)= \int_0^\infty \exp{-\lambda t}\hat G_t(s)\df t
\eqqcolon \int_0^\infty \exp{-\lambda t}\df \mathfrak G(t;s).
\]
Observe that $\lim_{t\to0} \mathfrak G(t;s)=: \fG(0;s)=0$ if there is no mass in $t=0$.
Firstly, we adapt the strategy used in \cite{CP2024} to establish a generalized uniform Tauberian result in $\lambda$ based on the assumption~\eqref{eq:hypLap}. Specifically, we aim to show that for any $c>0$
\begin{equation}\label{eq:TaubL_gen}
\lim_{t \to \infty} \sup_{0<s\leq \theta s_{\frac 1 t}}\left| \frac{\Gamma(\gamma+1)[\fG(ct;s)-\fG(0;s)]}{L(t)(ct)^{\gamma}\ell(1/s)s^{-\alpha}}-1\right|=0,\qquad\forall\,\theta>0.
\end{equation}
Then, we want to show that~\eqref{eq:TaubL_gen} implies
\begin{equation}\label{eq:MonotoneT}
\lim_{t \to \infty} \sup_{0<s\leq \theta s_{\frac 1 t}}\left| \frac{\Gamma(\gamma)\hat G_t(s)}{L(t)t^{\gamma-1}\ell(1/s)s^{-\alpha}}-1\right|=0,\qquad\forall\,\theta>0,
\end{equation}
in order to obtain a uniform control over $\hat G_t(s)$. Finally, we can conclude on $G_t(x)$ by means of~\cite[Thm.~1]{CP2024}. 

For the first statement, fix $c>0$ and define the function $j_c:[0,1]\to [0,\mathrm e^c]$ as
\[
j_c(x)\coloneqq \begin{cases}
   0 & \quad \textrm{for}  \quad x\in[0,\exp{-c}), \\
   1/x & \quad \textrm{for}  \quad x\in[\exp{-c},1] ,
\end{cases}
\]
and recall that, given $\lambda>0$, we have the identity
\begin{equation}\label{eq:Gwithj}
\fG(c/\lambda;s) -\fG(0;s)=\int_0^{\infty} \exp{-\lambda t}j_c(\exp{-\lambda t})\df\fG(t;s).
\end{equation}
It is known (see~\cite[Thm.~109]{Hardy} and~\cite[\S V.4, Lem.~4.1]{Widder}) that for any fixed $c>0$ and for all $\epsilon>0$, there exist polynomials $P_{\epsilon, c}(x)$ and $p_{\epsilon, c}(x)$ such that
\begin{equation}\label{eq:Weier1_c}
p_{\epsilon, c}(x)\le j_c(x)\le P_{\epsilon, c}(x)\, , \quad \quad \forall x\in[0,1],
\end{equation}
and
\begin{equation}\label{eq:Weier2_c}
\int_0^{\infty} \exp{-t} t^{\gamma-1} \left[P_{\epsilon, c}(\exp{-t})-p_{\epsilon, c}(\exp{-t}) \right]\df t \le \epsilon \Gamma(\gamma).
\end{equation}
From \eqref{eq:Weier1_c}, using the definition of $j_c(x)$ and the fact that $\fG(t;s)$ is non-decreasing in $t$, we have
\begin{subequations}
\begin{equation}\label{eq:ineq1_c}
\int_0^{\infty}  \exp{-\lambda t} p_{\epsilon, c}(\exp{-\lambda t})\df \fG(t;s) \le \fG(c/\lambda;s)-\fG(0;s)\le \int_0^{\infty}  \exp{-\lambda t} P_{\epsilon, c}(\exp{-\lambda t})\df \fG(t;s),
\end{equation}
\begin{equation}\label{eq:ineq2_c}
\gamma \int_0^{\infty}  \exp{-t}\,  t^{\gamma-1}\, p_{\epsilon,c}(\exp{-t}) \df t  \, \le c^\gamma \le \, \gamma\int_0^{\infty}  \exp{-t}\,t^{\gamma-1}\, P_{\epsilon,c}(\exp{-t})\df t,
\end{equation}
\end{subequations}
where we started from the identity $\gamma \int_0^\infty \exp{-t} t^{\gamma-1} j_c(\exp{-t})\df t = c^\gamma$ in the second line.
Combining~\eqref{eq:ineq1_c} with~\eqref{eq:ineq2_c} we get
\[
\left| \frac{\Gamma(\gamma+1)[\fG(c/\lambda;s) -\fG(0;s)]}{L(1/\lambda)\lambda^{-\gamma}\ell(1/s)s^{-\alpha}}-c^\gamma \right|\le \max\{A_{\lambda,\epsilon, c},B_{\lambda,\epsilon, c} \}\,,
\]
where
\begin{align*}
A_{\lambda,\epsilon, c}&\coloneqq\left| \frac{\Gamma(\gamma+1)\int_0^{\infty}  \exp{-\lambda t} P_{\epsilon, c}(\exp{-\lambda t})\df \fG(t;s)}{L(1/\lambda)\lambda^{-\gamma}\ell(1/s)s^{-\alpha}} 	- \gamma \int_0^{\infty}  \exp{-t}\,  t^{\gamma-1}\, p_{\epsilon, c}(\exp{-t}) \df t  \right|,\\
B_{\lambda,\epsilon, c}&\coloneqq 
\left| \frac{ \Gamma(\gamma+1)\int_0^{\infty}  \exp{-\lambda t} p_{\epsilon, c}(\exp{-\lambda t})\df \fG(t;s)}{L(1/\lambda)\lambda^{-\gamma}\ell(1/s)s^{-\alpha}}-  \gamma \int_0^{\infty}  \exp{-t}\,  t^{\gamma-1}\, P_{\epsilon, c}(\exp{-t}) \df t  \right|.
\end{align*}
Let $D_c(\epsilon),d_c(\epsilon)$ denote the maximum degrees and $(a_k(\epsilon))_{k=0}^{D_c(\epsilon)},(b_k(\epsilon))_{k=0}^{d_c(\epsilon)}$ the coefficients of the polynomials $P_{\epsilon, c}(x), p_{\epsilon, c}(x)$ respectively. 
We can therefore rewrite
\begin{align}
\int_0^{\infty}  \exp{-\lambda t} P_{\epsilon, c}(\exp{-\lambda t})\df \fG(t;s)&=\sum_{k=0}^{D_c(\epsilon)}a_k(\epsilon)\hat\fG((k+1)\lambda,s),\nonumber\\
\gamma \int_0^{\infty}\exp{-x}x^{\gamma-1}P_{\epsilon, c}(\exp{-x})\df x&=\Gamma(\gamma+1)\sum_{k=0}^{D_c(\epsilon)}a_k(\epsilon)(k+1)^{-\gamma},\label{eq:failAlpha0_c}
\end{align}
and similarly for $p_{\epsilon, c}(x)$. Jointly with \eqref{eq:Weier2_c}, these provide the bounds
\begin{align*}
A_{\lambda,\epsilon, c} & \leq  \Gamma(\gamma+1)
\sum_{k=0}^{D_c(\epsilon)}|a_k(\epsilon)|  (k+1)^{-\gamma} \left|\frac{\hat\fG((k+1)\lambda,s)}{L(1/\lambda)[(k+1)\lambda]^{-\gamma}\ell(1/s)s^{-\alpha}}-1\right| +\epsilon\Gamma(\gamma+1),\\
B_{\lambda,\epsilon, c} & \leq \Gamma(\gamma+1) 
\sum_{k=0}^{d_c(\epsilon)}|b_k(\epsilon)|  (k+1)^{-\gamma} \left|\frac{\hat\fG((k+1)\lambda,s)}{ L(1/\lambda)[(k+1)\lambda]^{-\gamma}\ell(1/s)s^{-\alpha}}-1\right| +\epsilon\Gamma(\gamma+1).
\end{align*}
Let $Dd_c(\epsilon)\coloneqq\max\{D_c(\epsilon),d_c(\epsilon)\}$ and observe that for every $k\in\{0,\dots,Dd_c(\epsilon)$ we can write
\begin{multline*}
\left|\frac{\hat\fG((k+1)\lambda,s)}{L(1/\lambda)[(k+1)\lambda]^{-\gamma}\ell(1/s)s^{-\alpha}} -1 \right|\\
\leq \left| \frac{L(1/[(k+1)\lambda])}{L(1/\lambda)}\right|\left|\frac{\hat\fG((k+1)\lambda,s)}{L(1/[(k+1)\lambda])[(k+1)\lambda]^{-\gamma}\ell(1/s)s^{-\alpha}} -1 \right| + \left| \frac{L(1/[(k+1)\lambda])}{L(1/\lambda)}-1\right|.
\end{multline*}
Given $\delta>0$, fix $\epsilon =\epsilon(\delta)< \delta / (3\Gamma(\gamma+1))$, and consequently the polynomials $P_{\epsilon, c},\,p_{\epsilon, c}$. Define
$K(\epsilon)\coloneqq \Gamma(\gamma+1) \max_k\{|a_k(\epsilon)|,|b_k(\epsilon)|\}$. 
Using the fact that $L$ is slowly varying and by hypothesis~\eqref{eq:hypLap}, for every $\theta>0$ we can find $\bar \lambda(\delta,\theta, c)\coloneqq \min_k\bar \lambda(\delta,\theta, c, k)$ such that for all $\lambda\leq\bar \lambda(\delta,\theta,c)$ and for every $ 0\leq k \leq Dd_c(\epsilon)$
\begin{align*}
 \left| \frac{L(1/[(k+1)\lambda])}{L(1/\lambda)}-1\right| &\leq  \frac\delta{3K(\epsilon)(Dd_c(\epsilon)+1)},\\
\sup_{0<s\leq \theta s_\lambda}\left| \frac{\hat\fG((k+1)\lambda, s)}{ L(1/[(k+1)\lambda]) [(k+1)\lambda]^{-\gamma}\ell(1/s)s^{-\alpha}}-1\right| &\leq  \frac\delta{3K(\epsilon)(Dd_c(\epsilon)+1)}\frac 1{1+\delta/[3K(\epsilon)(Dd_c(\epsilon)+1)]}.
\end{align*}
This ensures that, for all $\lambda\leq \bar\lambda(\delta,\theta, c)$:
\[
\sup_{0<s\leq \theta s_\lambda}\left| \frac{\Gamma(\gamma+1)[\fG(c/\lambda;s)-\fG(0;s)]}{L(1/\lambda)\lambda^{-\gamma}\ell(1/s)s^{-\alpha}}-c^\gamma\right|
\leq \sup_{0<s\leq \theta s_\lambda}\max\{A_{\lambda,\epsilon(\delta), c},B_{\lambda,\epsilon(\delta), c}\} \le \delta.
\]
Setting $t \coloneqq 1/\lambda$, dividing the expression inside the absolute value by the constant $c^\gamma$, and letting $\delta \to 0$ yield~\eqref{eq:TaubL_gen}.

In order to prove~\eqref{eq:MonotoneT}, let us assume that $\hat G_t(s)$ is eventually non-decreasing in $t$ uniformly for $s\in(0,\varepsilon)$, and fix $\epsilon>0$. By eventual monotonicity, for $t$ sufficiently large we have
\[
\fG(t;s)-\fG(t-\epsilon t;s)=\int_{t-\epsilon t}^t \hat G_y(s)\df y\leq \hat G_t(s)\epsilon t\leq\int_t^{t+\epsilon t} \hat G_y(s)\df y= \fG(t+\epsilon t;s)-\fG(t;s).
\]
Hence we have
\[
\left| \frac{\Gamma(\gamma)\hat G_t(s)}{L(t)t^{\gamma-1}\ell(1/s)s^{-\alpha}}-1\right|\leq \max\{A,B\}
\]
where
\[
 A \coloneqq \left| \frac{\Gamma(\gamma+1)[\fG(t+\epsilon t;s)-\fG(t;s)]}{\gamma \epsilon L(t)t^{\gamma}\ell(1/s)s^{-\alpha}}-1\right|, \quad
B\coloneqq \left| \frac{\Gamma(\gamma+1)[\fG(t;s)-\fG(t-\epsilon t;s)]}{\gamma \epsilon L(t) t^{\gamma}\ell(1/s)s^{-\alpha}}-1\right|.
\]
Thanks to~\eqref{eq:TaubL_gen}, we can now bound their suprema over the same domain $0 < s \le \theta s_{\frac 1 t}$. For the upper bound $A$, we obtain
\begin{align*}
\sup_{0<s\leq \theta  s_{\frac 1 t}} A
& \leq \frac 1 {\gamma\epsilon}\sup_{0<s\leq \theta s_{\frac 1 t}}\left| \frac{\Gamma(\gamma+1)[\fG(t;s) -\fG(0;s)]}{L(t)t^{\gamma}\ell(1/s)s^{-\alpha}}-1\right| \\
& \qquad +\frac{(1+\epsilon)^\gamma}{\gamma\epsilon} \sup_{0<s\leq \theta  s_{\frac 1 t}}\left| \frac{\Gamma(\gamma+1)[\fG(t(1+\epsilon);s)-\fG(0;s)]}{L(t)(t(1+\epsilon))^{\gamma}\ell(1/s)s^{-\alpha}}-1\right| \\
&\qquad +\frac 1 {\gamma\epsilon} \left|(1+\epsilon)^\gamma-1-\gamma \epsilon \right|.
\end{align*}
Taking $t$ large enough so that both suprema on the right-hand side are $O(\epsilon^2)$ by~\eqref{eq:TaubL_gen} (applied with $c=1$ and $c=1+\epsilon$ respectively), we obtain $\sup_{0<s\leq \theta  s_{\frac 1 t}} A \le O(\epsilon)$. 
By a similar argument, setting $c=1$ and  $c=1-\epsilon$, for $t$ large enough we have
\begin{align*}
\sup_{0<s\leq \theta  s_{\frac 1 t}} B 
& \leq \frac 1 {\gamma\epsilon}\sup_{0<s\leq \theta s_{\frac 1 t}}\left| \frac{\Gamma(\gamma+1)[\fG(t;s) -\fG(0;s)]}{L(t)t^{\gamma}\ell(1/s)s^{-\alpha}}-1\right| \\
& \qquad +\frac{(1-\epsilon)^\gamma}{\gamma\epsilon} \sup_{0<s\leq \theta  s_{\frac 1 t}}\left| \frac{\Gamma(\gamma+1)[\fG(t(1-\epsilon);s) -\fG(0;s)]}{L(t)(t(1-\epsilon))^{\gamma}\ell(1/s)s^{-\alpha}}-1\right|\\
&\qquad +\frac 1 {\gamma\epsilon} \left|1-\gamma \epsilon-(1-\epsilon)^\gamma \right|\\
&  = O(\epsilon).
\end{align*}
Letting $\epsilon \to 0$ concludes the proof of~\eqref{eq:MonotoneT}.
\end{proof}

\begin{proof}[Proof of Corollary~\ref{cor:Lap}]
Refer to the proof of~\cite[Thm.~XIII.5.5]{Feller2}. By defining
\[
\hat G_t(s)\coloneqq \hat G_n(s)\qquad\text{for}\quad n\leq t<n+1,
\]
the Laplace transform of $\hat G_t(s)$ is given by
\begin{equation}\label{eq:cor1}
\hat {\mathfrak G}(\lambda,s)
= \sum_{n=0}^\infty\hat G_n(s)\int_n^{n+1}\exp{-\lambda t}\df t
 = \frac{1-\exp{-\lambda}}\lambda\hat {\mathcal G}(\exp{-\lambda},s),
\end{equation}
and the conclusion immediately follows by virtue of Theorem~\ref{thm:doubleLaplace}.
\end{proof}

\appendix 
\section{Technical results}
In the main text, we make extensive use of the following elementary equivalence.
\begin{lemma}\label{lem:eqPointUnif}
Let $(f_\lambda(s))_{\lambda> 0}$ be a family of functions where $f_\lambda:(0,\infty)\to[0,\infty)$ for each $\lambda> 0$.  Let $(s_\lambda)_{\lambda> 0}$ be such that $s_\lambda \to 0$ as $\lambda\to 0$. The following statements are equivalent:
\begin{enumerate}[label=(\roman*)]
\item 
$ \displaystyle
\lim_{\lambda\to 0} \sup_{0<s\leq s_\lambda} f_\lambda(s)=0.
$
\item For any $(\tilde s_\lambda)_{\lambda> 0}$ such that $\tilde s_\lambda\leq s_\lambda$ for all $\lambda> 0$, we have
$ \displaystyle
\lim_{\lambda\to 0}f_\lambda(\tilde s_\lambda)=0.
$
\end{enumerate}
\end{lemma}
\begin{proof}
We prove the two implications.
\paragraph{$(i)\implies (ii)$.} We assume that $(i)$ holds. Let $(\tilde s_\lambda)_{\lambda> 0}$ be such that $\tilde s_\lambda\leq s_\lambda$ for all $\lambda> 0$. Then, for each $\lambda$, we have
\[
0\leq f_\lambda(\tilde s_\lambda)\leq \sup_{0<s\leq s_\lambda} f_\lambda(s),\quad \text{which implies}\quad  0\leq \limsup_{\lambda \to 0} f_\lambda(\tilde s_\lambda)\leq \lim_{\lambda\to 0}\sup_{0<s\leq s_\lambda} f_\lambda(s)=0.
\]
\paragraph{$(ii)\implies (i)$.}  We assume that $(ii)$ holds. By contradiction, we suppose that $(i)$ does not hold, namely that there exist $\epsilon>0$ and a subsequence $(\lambda_n)_{n\in\N}$, with $\lambda_n\to 0$ as $n\to\infty$, such that $\sup_{0<s\leq s_{\lambda_n}} f_{\lambda_n}(s)\geq 2\epsilon$ for all $n\in\N$.
Hence, there exists a sequence $(\bar s_n)_{n\in\N}$ such that $\bar s_n\leq s_{\lambda_n}$ and $f_{\lambda_n}(\bar s_n)\geq \epsilon$. Define
\[
\tilde s_\lambda\coloneqq \begin{cases}
\bar s_n\qquad&\text{if}\quad \lambda=\lambda_n,\\
s_\lambda\qquad&\text{otherwise}.
\end{cases}
\]
Then $\tilde s_\lambda\leq s_\lambda$ for all $\lambda$, but $f_{\lambda_n}(\tilde s_{\lambda_n}) \geq \epsilon$, so $\limsup_{\lambda\to0} f_\lambda(\tilde s_\lambda)\geq \epsilon$, contradicting the assumption~$(ii)$.
\end{proof}
\begin{APPcorollary}\label{cor:eqPointUnif}
Let $(f_\lambda(s))_{\lambda> 0}$ be a family of functions where $f_\lambda:(0,\infty)\to[0,\infty)$ for each $\lambda>0$.  Let $(s_\lambda)_{\lambda>0}$ be such that $s_\lambda \to 0$ as $\lambda \to 0$. The following statements are equivalent:
\begin{enumerate}[label=(\roman*)]
\item For all \( \theta > 0 \), $\lim_{\lambda\to0} \sup_{0<s\leq\theta s_\lambda} f_\lambda(s) = 0$.
\item For every $\theta>0$, and for any \( (\tilde s_\lambda)_{\lambda> 0} \) such that \( \tilde s_\lambda \leq \theta s_\lambda\) for all $\lambda>0$, we have
$\lim_{\lambda \to 0} f_\lambda(\tilde s_\lambda) = 0$.
\end{enumerate}
\end{APPcorollary}

Moreover, for the sake of readability, we recall the uniform version~\cite[Lem.~1]{CP2024} of the Monotone Density Theorem; see~\cite[Thm.~1.7.2 and Thm.~1.7.5]{bingham} for the classical result.
\begin{lemma}\label{lem:MDT}
Let $(u_t(x))_{t\geq 0}$ be a one-parameter family of locally bounded functions that are non-increasing on $\R^+$ and such that $\lim_{x\to-\infty}u_t (x)=0$.
Suppose that there exist $\sigma>0$, $\alpha\in (0,1)$ and $(x_t)_{t\geq 0}$, with $x_t\xrightarrow{t\to\infty}\infty$, such that
\[
\lim_{t\to\infty}\sup_{x\geq x_t}\left|\frac{\int_{-\infty}^x u_t(y)\df y}{L(t) t^\sigma\ell(x)x^{\alpha}}-1 \right| =0,
\]
where $L, \ell$ are slowly varying functions at infinity. Then
\[
\lim_{t\to\infty}\sup_{x\geq (1+\eta)x_t}\left|\frac{u_t(x)}{\alpha L(t)t^\sigma\ell(x)x^{\alpha-1}}-1 \right| =0\qquad \forall\, \eta>0.
\]
Moreover, the same conclusion holds for a family of functions $(u_t(x))_{t\geq 0}$ non-decreasing on $\R^+$ with $\alpha\in(-\infty,0)\bigcup[1,\infty)$.
\end{lemma}

\section{More on Random Sums}\label{app:RS}
In this Appendix we clarify the  relationship between Examples~\ref{ex:IID} and~\ref{ex:RS}. Consider the control-cost process~\eqref{eq:CCprocess} introduced in Section~\ref{sec:RS}, where the increments $X_k$ and $Y_k$ are independent, and  define the counting (renewal) process $({\mathrm N}(t))_{t\geq 0}$ induced by the control process as ${\mathrm N}(t)\coloneqq \sup\{ n\in\N_0\,:\, S_n\leq t\}$. 
Denote by $\hat{\mathfrak h}_\textrm{count}(\lambda,s)$ the double Laplace transform associated with $(C_{{\mathrm N}(t)})_{t\geq 0}$ and by $\hat{\mathfrak h}(\lambda,s)$ the transform associated with $(C_{\tau(t)})_{t\geq 0}$. Since $\tau(t)=\mathrm{N}(t)+1$ and by independence, it follows that\footnote{We combine $\hat {\mathfrak h}(\lambda,s)=\hat p_Y(s)\hat{\mathfrak h}_\text{count}(\lambda,s)$ with~\eqref{eq:doubleRS}, recalling that $\hat {\mathfrak h}(\lambda,s)=\lambda^{-1}-s\hat{\mathfrak G}(\lambda,s)$ and similarly for $\hat{\mathfrak h}_\text{count}$.}
\begin{equation}\label{eq:count}
\hat{\mathfrak G}_\textrm{count}(\lambda,s)
= \hat p_X(\lambda) \hat{\mathfrak G}(\lambda,s),
\end{equation}
and thus, under the assumptions of Example~\ref{ex:RS}, we can write
\begin{equation}
\hat{\mathfrak G}_\textrm{count}(\lambda,s)= -\frac 1 {s\lambda} \sum_{k=1}^\infty \left( \frac{\hat p_X(\lambda)}{1-\hat p_X(\lambda)} a(s)\right)^k.
\end{equation}
By recalling the $\hat \cG(z,s)$ obtained in~\eqref{eq:doubleIID}
\begin{align*}
\hat \cG(z,s)
=\frac1{s(z-1)}\sum_{k=1}^\infty \left(\frac z {1-z}a(s)\right)^k,
\end{align*}
we immediately recognize that
\begin{align}\label{eq:IIDfromRS}
\hat {\mathfrak G}_\textrm{count} (\lambda,s)
= \frac{1-\hat p_X(\lambda)}\lambda \hat \cG(\hat p_X(\lambda),s).
\end{align}
The i.i.d.\ model of Example~\ref{ex:IID} can be now derived as a special case of~\eqref{eq:IIDfromRS} by choosing $p_X(x)=\delta(x-1)$, i.e., a degenerate deterministic distribution. Indeed, with this choice $\hat p_X(\lambda)=\exp{-\lambda}$ and, in agreement with the proof of Corollary~\ref{cor:Lap} (see~\eqref{eq:cor1}), we directly recover~\eqref{eq:doubleIID}.
We can therefore conclude that~\eqref{eq:count} is a generalization of~\eqref{eq:doubleIID}, with $\hat p_X(\lambda)$ acting the role of $z$.

Finally, we can  provide a rigorous comparison with the results of~\cite[\S~3.2.2]{CP2024} for generic non-deterministic $\hat p_X(\lambda)$. Starting from~\eqref{eq:count} and~\eqref{eq:RSbeta}, we have
\begin{align*}
\hat {\mathfrak G}_\textrm{count} (\lambda,s)
& = -\frac {a(s)}{s\lambda}\frac{\hat p_X(\lambda)}{1-\hat p_X(\lambda)}\left[1+\sum_{k=1}^\infty [a(s)]^k \left(\frac{\hat  p_X(\lambda)}{1-\hat p_X(\lambda)}\right)^k \right].
\end{align*}
 which can be equivalently expressed by performing a termwise inversion in $\lambda$, justified by the absolute convergence of the series and the decoupling from $a(s)$ at all orders.
It is not difficult to recognize that the leading-order coefficient $\frac{\hat p_X(\lambda)}{\lambda[1-\hat p_X(\lambda)]}$ becomes $\E({\mathrm N}(t))$~\cite[Eq.~(3.4)]{GL2001}, and similarly the remaining terms in the time domain can be controlled by the higher-order moments of ${\mathrm N}(t)$. Hence, we simply traced back\footnote{To make this comparison fully explicit, observe that $\hat {\mathfrak h}_\textrm{count} (\lambda,s) = \int_0^\infty \exp{-\lambda t} \phi_{{\mathrm N}(t)}(\hat p_Y(s))\df t$, where $\phi_{{\mathrm N}(t)}(z)\coloneqq \sum_{n=0}^\infty \P({\mathrm N}(t)=n) z^n$, as defined in~\cite{CP2024}.} to the analysis in~\cite[Asm.~2]{CP2024}, analogously to Example~\ref{ex:IID}.

\section{Two-sided shrinking window for the Tauberian region}\label{app:BoundedRegion}
The proof of Theorem~\ref{thm:doubleLaplace} can be easily adapted to a two-sided shrinking interval for the transform parameter \(s\). More precisely, choose two families $(s_\lambda^*)_{\lambda>0}$ and $(s_\lambda)_{\lambda>0}$  such that $0<s_\lambda^*<s_\lambda$ for all $\lambda >0$ and $s_\lambda\to 0$ as $\lambda \to 0$. If the hypothesis~\eqref{eq:hypLap} is modified so that the supremum is taken over the shrinking interval 
\[
s \in [\theta \, s_{\lambda}^*, \, \theta \, s_{\lambda}],\qquad \forall\theta > 0,
\]
then the statement of the theorem remains valid with a corresponding change for the range of the spatial variable
\[
x \in \left[1/s_{\frac 1t}, \, 1/s_{\frac 1 t}^*\right].
\]
In this sense, a large deviation principle  can hold in a region that diverges to infinity  but remains localized, and can be derived via a properly modified uniform Tauberian theorem.
The same observation readily applies to  the single-parameter framework of~\cite{CP2024}.
As an illustration of the utility of this regime, we will shortly demonstrate, through a simple example,  how  uniform control on a bounded $s$-interval can capture a
non-trivial large deviation result for a family of  distributions with sharp tails.
\newline

Consider the following family of truncated distributions with  truncation threshold that tends to infinity. For every $t>1$, let $X_t$ be a truncated Pareto random variable with probability density function 
\[
p_{X_t}(x)=
\begin{cases}
\frac{\beta}{ c_t} x^{-\beta-1}\qquad & \text{if}\quad x\in [1,t],\\
0\qquad &\text{otherwise},
\end{cases} \qquad\text{with}\quad c_t\coloneqq 1-t^{-\beta},\quad\beta\in(0,1).
\]
Its Laplace transform is
\[
\hat p_{X_t}(s)= \frac\beta{c_t} \int_1^t \exp{-sx} x^{-\beta-1}\,\df x
=\frac{\beta}{c_t} s^\beta [\Gamma(-\beta, s)-\Gamma(-\beta, ts)],
\]
where $\Gamma(-\beta,s)$ denotes the upper incomplete gamma function.
In the standard large deviation limit, given that $\P(X_t>x)=0$ if $x>t$, the moments of all orders exist and the Laplace transform must be an analytic function.
In fact, in the scaling regime $x\gg t$ as $t\to \infty$ , or equivalently $s\ll 1/t$, we observe that for all $(\tilde s_t)_{t\geq 0}$ such that $\tilde s_t\ll 1$ and $t \tilde s_t\ll1$ we have
\begin{align*} 
\hat p_{X_t}(\tilde s_t)
& = 1-\E_t(X) \tilde s_t+O(\tilde s_t^2)\quad\text{as}\quad t \to \infty,\quad\text{with}\quad\E_t(X)=\frac \beta {1-\beta} \frac {t^{1-\beta}-1}{c_t}.
\end{align*}
If we consider instead the intermediate region $1\ll x\ll t$ with $t\to \infty$, we must provide the asymptotic expansion of the incomplete gamma functions for $\tilde s_t\ll 1$ but with $t\tilde s_t\gg1$. Hence for $t\to \infty$ we get
\begin{align}
\hat p_{X_t}(\tilde s_t)
& = \left(\frac 1 {c_t}+\frac 1 {c_t} \beta \Gamma(-\beta)\tilde s_t^\beta+\frac 1 {c_t}\frac\beta{1-\beta} \tilde s_t+o(\tilde s_t)\right)-\left(\frac \beta {c_t}\frac{t^{-\beta-1}\exp{-t\tilde s_t}}{\tilde s_t}+o\left(t^{-\beta-1}\exp{-t\tilde s_t} \tilde s_t^{-1}\right)\right).\label{eq:truncPar}
\end{align}
We now apply the uniform Tauberian theorem~\cite[Thm.~1]{CP2024} to (similarly for $\beta\in(1,2)$, as in Example~\ref{ex:IID})
\(
\hat G_t(s)\coloneqq\frac{1-\hat p_{X_t}(s)}s,
\)
noting that by~\eqref{eq:truncPar} and Corollary~\ref{cor:eqPointUnif} the following assumption is satisfied
\[
\forall \theta>0,\qquad \lim_{t\to\infty} \sup_{s\in [\theta s_{t}^* , \theta s_{t}]}\left|\frac{c_t\hat G_t(s)}{-\beta\Gamma(-\beta)s^{\beta-1}}-1 \right|=0,
\]
where $s_{t}^*\gg 1/t$ and $s_{t}\ll 1$ as $t\to\infty$.
By Lemma~\ref{lem:MDT}, we finally recover
\[
\lim_{t\to\infty} \sup_{x\in [ 1/s_{t} , 1/s^*_{t}]}\left|\frac{\P(X_t>x)}{x^{-\beta}/c_t}-1 \right|=0.
\]
\newline

We next consider the model of Corollary~\ref{cor:Leap}, focusing on a different scaling window. In the complementary regime $1 \ll x \ll t$, corresponding to $t/x_t \to \infty$, one obtains
\[
\lim_{\lambda \to 0}\sup_{\theta s_\lambda \leq s < \varepsilon}
\left| \lambda s \hat{\mathfrak G}(\lambda, s) - 1 \right| = 0,
\qquad \forall \theta > 0,
\]
with $\varepsilon > 0$ arbitrarily small and $s_\lambda$ as in Corollary~\ref{cor:Leap}. By definition~\eqref{eq:doubleCost}, this immediately yields (as expected)
\[
\int_0^{x} h_t(c)\, dc \longrightarrow 0 \qquad \text{as } t \to \infty,\quad \text{for all }1\ll x\ll t,
\]
so that no Tauberian argument is required in this regime.

\section{An explicitly solvable example and further remarks}\label{app:ML}
In the setting of  Corollary~\ref{cor:Leap}, we  consider a special case, corresponding to Mittag-Leffler waiting times, following~\cite{BGS2023}. With this choice, not only does the joint transform admit an explicit inversion with respect to $\lambda$, but also the general framework reduces to a setting previously analysed in~\cite{CP2024} thereby providing a useful consistency check of  the results.

\begin{example}\label{ex:ML}
Let $p_X(x)$ be the Mittag--Leffler waiting time distribution, namely
\begin{equation}\label{eq:ML}
\hat p_X(s)=\frac 1{1+s^\beta},\quad\text{with}\quad \beta\in(0,1).
\end{equation}
Then
\begin{align*}
\hat h_t(s)= \mathcal L_\lambda^{-1}\{\hat {\mathfrak h}(\lambda,s)\}(t)
&=1-\frac{s^\beta}{1+s^\beta}\left[ 1+\cL^{-1}_\lambda \left\{\frac 1{\lambda(s+\lambda)^\beta}\right\}(t)\right],
\end{align*}
and recalling that the product of two Laplace transforms provides the convolution of the corresponding inverse Laplace transforms
\begin{align*}
\cL_{L_t}(s)&=\exp{st}\hat h_t(s)=\exp{st}\left[ 1-\frac{s^\beta}{1+s^\beta}\left(1+\int_0^t\frac{x^{\beta-1}\exp{-sx}}{\Gamma(\beta)}\df x \right)\right]\\
&=\left(1+\sum_{m=1}^\infty\frac{(st)^m}{m!}\right)\left[ 1-\frac{s^\beta}{1+s^\beta}\left( 1+\frac{t^\beta}{\Gamma(\beta+1)}+\sum_{m=1}^\infty (st^\beta)^m \frac{\Gamma(\beta+m)}{\Gamma(\beta)\Gamma(\beta+m+1)}\right)\right].
\end{align*}
Since we have direct control over $\mathcal L_{L_t}(s)$, we can apply both the standard and the uniform Tauberian theorems to obtain, respectively, the tail asymptotics and the large deviation result. For the latter, a direct computation using~\cite[Thm.~1]{CP2024} on the family\footnote{In particular, note that the error $\mathcal E_t(s)$ consists of terms that are monotone non-decreasing in their argument, hence taking the $\sup_{s\leq s_t}$ is equivalent to evaluating at~$s= s_t$.}
\[
\cL_{L_t}(s)=1- \frac{s^\beta t^\beta}{\Gamma(\beta+1)}+\mathcal E_t(s) \quad\text{with}\quad \lim_{t\to\infty}\sup_{s\in(0,\theta s_t]}\frac{|\mathcal E_t(s)|}{s^\beta t^\beta}=0\quad \forall \theta>0,
\]
where $(s_t)_{t\geq 0}$ is such that $t\cdot s_t\to 0$ as $t\to\infty$,
reproduces a special case of the statement of Corollary~\ref{cor:Leap}, whereas for the former we can verify consistency with the well-known result in the literature~\cite[Eq.~(26)]{Koren2007}\cite{Eliazar2004}. For fixed $t\geq0$, we can write
\[
\cL_{L_t}(s)= 1-s^\beta\left(1+\frac{ t^\beta}{\Gamma(\beta+1)}\right)+o(s^\beta)\quad\text{as}\quad s\to 0,
\]
which implies~\cite[Thm.~1.7.1 and Thm.~1.7.2]{bingham}
\begin{equation}\label{eq:BJ}
\P(L_t>x)\sim \frac {[t^\beta+\Gamma(\beta+1)] x^{-\beta}}{\Gamma(\beta+1)\Gamma(1-\beta)}\quad\text{as}\quad x\to \infty.
\end{equation}
In particular, by taking the derivative we recover the well-known probability density function
\[
\frac{\sin(\pi\beta)}\pi \frac{[t^\beta+\Gamma(\beta+1)]}{x^{\beta+1}},
\]
and in the limit $t\to 0$ the leapover (or first-ladder height) for a one-sided random walk coincides with the first increment $X_1$, indeed $\P(L_t>x)\sim x^{-\beta}/\Gamma(1-\beta)$ as $x\to\infty$ according to~\eqref{eq:ML}.
\end{example}

Observe that the relevance of the uniform Tauberian method stems from the fact that the main difficulty is not the mere existence of a large deviation region, but rather the identification of its precise form. 
In this respect, our uniform Tauberian technique offers an alternative route to the classical probabilistic approach, for explicitly identifying the threshold sequence $(x_t)_{t\geq 0}$, by moving to the Laplace--Stieltjes domain.
To illustrate this distinction more concretely, we leverage on Example~\ref{ex:ML}: the  family $(L_t)_{t\ge0}$  satisfies the following fixed-parameter asymptotics
\begin{equation}\label{eq:BJtoLD}
\forall t\geq 0, \quad \forall \delta>0,\quad \exists x_{t,\delta}\xrightarrow{\delta\to 0}{\infty} \quad\text{such that}\quad \\
\left| \frac{\P(L_t> x_{t,\delta})}{C_t  x_{t,\delta}^{-\beta}}-1\right|<\delta,
\end{equation}
with $C_t=\frac{[t^\beta+\Gamma(\beta+1)]}{\Gamma(\beta+1)\Gamma(1-\beta)} $.
As the result holds for any $\delta=\delta(t)\to 0$ as $t\to\infty$, this provides the existence of a large deviation region by means of the single-parameter equivalent of Lemma~\ref{lem:eqPointUnif}, that is~\cite[Lem.~3]{CP2024}; see also~\cite[Thm.~1.1]{ClineHsing}.
By means of our uniform Tauberian theorems, we can explicitly identify a large deviation region (defined by $x \ge x_t$ and $t/x_t \to 0$ as $t \to \infty$, as seen in Example 3) that cannot otherwise be extracted using only the information in \eqref{eq:BJtoLD}.

As a final comment, we observe that the shrinking window introduced in Appendix~\ref{app:BoundedRegion} breaks the  correspondence between the univariate tail asymptotics and the large deviation framework. In the former, the relevant region is necessarily a half-line; however, in the large deviation setting, the presence of an increasing parameter allows for a different decomposition of the asymptotic regime. 

\bigskip
{\bf Acknowledgements}
This research is part of the authors' activity within the UMI Group ``DinAmicI'' (www.dinamici.org).
The research is part of GC's activity within GNFM/INdAM.

\bigskip

\appendix

 \bibliographystyle{plain}
\bibliography{biblio}

\end{document}